\documentclass{amsart}
\usepackage{amssymb}
\usepackage{epsfig}  
\usepackage{epic,eepic}       
\usepackage[bottom]{footmisc}
\usepackage[pagebackref]{hyperref}
\newtheorem{thm}{Theorem}[section]
\newtheorem{prop}[thm]{Proposition}

\theoremstyle{definition}

\newtheorem{remark}[thm]{Remark}
\numberwithin{equation}{section}
 
\DeclareMathOperator{\GL}{GL}
\DeclareMathOperator{\Gal}{Gal}

\begin{document}

\title{Torsion groups of elliptic curves over quadratic fields $\mathbb{Q}(\sqrt{d}),$ $0<d<100$}

\author{Antonela Trbovi\'c}
\address{Department of Mathematics, University of Zagreb, Bijeni\v{c}ka
cesta 30, 10000 Zagreb, Croatia}
\email{antonela.trbovic@gmail.com}

\begin{abstract}
We prove results towards classifying the possible torsion
subgroups of elliptic curves over quadratic fields $\mathbb{Q}(\sqrt{d})$, where $0<d<100$ is a square-free integer, and obtain a complete classification for 49
out of 60 such fields. Over the remaining
11 quadratic fields, we cannot rule out the possibility of the group $\mathbb{Z}/16\mathbb{Z}$
appearing as a torsion group of an elliptic curve.
\end{abstract}



\thanks{The author was supported by the QuantiXLie Centre of Excellence, a project co-financed by the Croatian
	Government and European Union through the European Regional Development Fund - the Competitiveness and
	Cohesion Operational Programme (Grant KK.01.1.1.01.0004).
}

\maketitle

\section{Introduction}

For an elliptic curve $E$ defined over a number field $K,$ let $E(K)$ be the set of all $K-$rational points on $E.$ By the Mordell-Weil theorem and the structure theorem for finitely generated abelian groups, we know that $E(K)$ is isomorphic to $E(K)_{tors} \oplus \mathbb{Z}^r,$ where $r$ is a non-negative integer and $E(K)_{tors}$ is the torsion subgroup.

In the case of $K=\mathbb{Q},$ Mazur's theorem \cite{BM} gives us all possible torsion subgroups of $E(\mathbb{Q})$:
$$\mathbb{Z}/n\mathbb{Z}, \:\:\: n=1,...,10,12,$$
$$\mathbb{Z}/2\mathbb{Z} \oplus \mathbb{Z}/2n\mathbb{Z}, \:\:\: n=1,2,3,4.$$

We also have a similar result by Kamienny, Kenku and Momose \cite{SK,MAKFM} concerning possible torsion subgroups of elliptic curves defined over any quadratic field, which are the following 26 groups:
$$\mathbb{Z}/n\mathbb{Z}, \:\:\: n=1,...,16,18,$$
$$\mathbb{Z}/2\mathbb{Z} \oplus \mathbb{Z}/2n\mathbb{Z}, \:\:\: n=1,...,6,$$
$$\mathbb{Z}/3\mathbb{Z} \oplus \mathbb{Z}/3n\mathbb{Z}, \:\:\: n=1,2,$$
$$\mathbb{Z}/4\mathbb{Z} \oplus \mathbb{Z}/4\mathbb{Z}.$$
While this theorem settles the question on what are the possibilities for the torsion subgroup over all quadratic fields, we are interested in what happens when we fix a quadratic field. In order to see what happens over a fixed field, one would have to go through each of the 26 groups mentioned above and check whether that is a possible torsion subgroup or not.

We are going to see why every group mentioned in Mazur's theorem has to appear as a possible torsion subgroup over all quadratic fields, and what happens with the groups $\mathbb{Z}/3\mathbb{Z}\oplus \mathbb{Z}/3n\mathbb{Z}, \: n=1,2, \: \mathbb{Z}/4\mathbb{Z}\oplus \mathbb{Z}/4\mathbb{Z}.$
For the rest of the groups we will follow the methods described in \cite{SKFN}.

From now on, let $K$ be a fixed quadratic field.
Let $Y_1(m,n)$ be the affine modular curve whose every $K-$rational point corresponds to an isomorphism class of an elliptic curve together with an $m-$torsion point $P_m\in E(K)$ and an $n-$ torsion point $P_n\in E(K)$ such that $P_m$ and $P_n$ generate a subgroup isomorphic to $\mathbb{Z}/m\mathbb{Z} \oplus \mathbb{Z}/n\mathbb{Z},$ and let $X_1(m,n)$ be its compactification (the same curve with adjoined cusps). We denote $X_1(1,n)$ by $X_1(n).$

More precisely, what we need to do in order to determine whether an elliptic curve with torsion $\mathbb{Z}/n\mathbb{Z} \oplus  \mathbb{Z}/m\mathbb{Z}$ over $K$ exists, for the rest of the 26 groups, is to determine whether there are $K-$rational points on $X_1(m,n)$ that are not cusps. These modular curves are either elliptic or hyperelliptic.

If the modular curve $X_1(m,n)$ is elliptic, we compute its rank. If the rank is positive, there are infinitely many elliptic curves over $ K $ with the given torsion subgroup, as the number of cusps is finite. If the rank is 0, we have to compute the torsion subgroup and check whether any torsion point corresponds to a $ K- $rational point on the modular curve that is not a cusp.

If the modular curve $X_1(m,n)$ is hyperelliptic, we compute the rank of the Jacobian of the curve. If the rank is 0, we also have to check whether any torsion point arises from a $ K- $rational point on the modular curve that is not a cusp. If the rank is positive, the problem becomes more difficult. More about this can be found in \cite{SKFN}.

\section{Groups from Mazur's theorem and \\ $\mathbb{Z}/3\mathbb{Z}\oplus \mathbb{Z}/3n\mathbb{Z}, \: n=1,2, \: \mathbb{Z}/4\mathbb{Z}\oplus \mathbb{Z}/4\mathbb{Z}$}
\label{sec:two}

In this section, we are going to show that every group mentioned in Mazur's theorem has to appear as a possible torsion group over any quadratic field $K.$

Let $E$ be an elliptic curve and denote by $\rho_{E,n}:\Gal(K/\mathbb{Q}) \rightarrow \GL_2(\mathbb{Z}/n \mathbb{Z})$ the $mod \: n$ Galois representation attached to $E.$

If ${P, P'}$ is a basis for $E[n],$ the subgroup of $E$ of points of order $n,$ and if $P$ is a point of order $n$ in $E(\mathbb{Q}),$ then $\rho_{E,n}(\sigma)=\begin{bmatrix} 
    1       & b \\
    0       & d
\end{bmatrix}, \: b, \in\mathbb{Z}/n\mathbb{Z}, d \in (\mathbb{Z}/n\mathbb{Z})^{\times}$  for every $\sigma \in \Gal(K/\mathbb{Q}),$ with respect to the basis $\{P,P'\}.$

We define a subgroup of $\GL_2(\mathbb{Z}/n\mathbb{Z}),$
$$\Gamma_1(n)=\left\{\begin{bmatrix}
1 & b \\ 0 & d
\end{bmatrix}: b \in \mathbb{Z}/n\mathbb{Z}, d \in
(\mathbb{Z}/n\mathbb{Z})^{\times}\right\},$$
which corresponds to $ X_1(n), $ i.e. the $mod \:n $ representations of elliptic curves parameterized by the points on $ X_1(n) $ are elements in $ G(n), $ with an appropriate choice of basis. Similarly, we define a subgroup 
 $$\Gamma_1(2,2n)=\left\{\begin{bmatrix}
1 & b \\ 0 & d
\end{bmatrix}: b \in 2\mathbb{Z}/2n\mathbb{Z}, d \in (\mathbb{Z}/2n\mathbb{Z})^{\times}\right\},$$ which corresponds to $ X_1(2,2n), $ in the sense described above.

For any of the groups $\mathbb{Z}/n\mathbb{Z}$ or $\mathbb{Z}/2\mathbb{Z} \oplus \mathbb{Z}/2n\mathbb{Z}$ appearing in Mazur's theorem, we
have that the corresponding modular curve $X_1(n)$ or $X_1(2,2n)$,
respectively, is of genus 0. Now, with $ X_G=X_1(n) $ or $ X_G=X_1(2,2n) $ in \cite[Lemma 3.5]{DJZ}, using the same arguments as in the proof
of  the lemma, but taking the base field to be a quadratic field $ K $
instead of $\mathbb{Q}$, we have that there are infinitely many elliptic curves $E$ over $K$
such that $\rho_{E,n}(\Gal (K/\mathbb{Q}))$ is conjugate (not just contained) in $\GL_2(\mathbb{Z}/n\mathbb{Z})$
to $\Gamma_1(n)$ or $\Gamma_1(2,2n)$, respectively, proving our claim.

Now, we will focus on the groups
$$\mathbb{Z}/3\mathbb{Z}\oplus\mathbb{Z}/3n\mathbb{Z}, \:\: n=1,2,$$
$$\mathbb{Z}/4\mathbb{Z}\oplus\mathbb{Z}/4\mathbb{Z}.$$
From the properties of the Weil pairing, we know that
$\mathbb{Z}/n\mathbb{Z} \times \mathbb{Z}/n\mathbb{Z} \subset E(K)$ only if $\mathbb{Q}(\zeta_n) \subset K.$ 

Hence, $\mathbb{Z}/3\mathbb{Z} \times \mathbb{Z}/3n\mathbb{Z} \subset E(K),$ $n=1,2,$ only when $K\supset\mathbb{Q}(\zeta_3)=\mathbb{Q}(\sqrt{-3})$ and $\mathbb{Z}/4\mathbb{Z} \times \mathbb{Z}/4\mathbb{Z} \subset E(K)$ only if $K \supset \mathbb{Q}(i).$ Moreover, the mentioned groups are the only groups, except the ones from the Mazur's theorem, that appear over $\mathbb{Q}(\sqrt{-3})$ and $\mathbb{Q}(i),$ respectively \cite{FN1,FN2}.

\section{Torsion over $\mathbb{Q}(\sqrt{17})$}

We will now demonstrate how to carry out the mentioned methods over the quadratic field $\mathbb{Q}(\sqrt{17}).$ We chose $\mathbb{Q}(\sqrt{17})$ because $17 \equiv 1 \pmod{8},$ and the significance of this relation will be clear later on.

As stated above, all groups from Mazur's theorem are possible torsion subgroups over $\mathbb{Q}(\sqrt{17}),$ and the groups $\mathbb{Z}/3\mathbb{Z}\oplus\mathbb{Z}/3n\mathbb{Z}, \: n=1,2,$ $\mathbb{Z}/4\mathbb{Z}\oplus\mathbb{Z}/4\mathbb{Z}$ are not.

For the rest of the groups we will follow the methods from \cite{SKFN} described above.

The equations for $X_1(m,n)$ can be found in \cite{HB,FPR}.

All computations in the following propositions will be done in Magma \cite{MAGMA}. Computations for this paper can be found  \href{https://web.math.pmf.unizg.hr/~atrbovi/magma.txt}{here}.

\begin{prop}\label{prop1}
There are infinitely many elliptic curves with torsion $\mathbb{Z}/11\mathbb{Z}$ over $\mathbb{Q}(\sqrt{17}).$
\end{prop}
\begin{proof}
To show this, we have to prove that the modular curve $X_1(11)$ defined over $\mathbb{Q}(\sqrt{17})$ has infinitely many points. It will suffice to see that the rank is positive, since the number of cusps on $X_1(11)$ is finite. For the modular curve $$X_1(11): \:\: y^2-y=x^3-x^2,$$ we compute $$rank(X_1(11)(\mathbb{Q}(\sqrt{17})))=1$$ in Magma. Now we can conclude that there are infinitely many elliptic curves with torsion $\mathbb{Z}/11\mathbb{Z}$ over $\mathbb{Q}(\sqrt{17}).$

We can also compute a generator of the group $X_1(11)(\mathbb{Q}(\sqrt{17}))$ (modulo the torsion subgroup), which is $$\left(\dfrac{1}{8}(-\sqrt{17}+1),\dfrac{1}{16}(\sqrt{17}+7)\right).$$
\end{proof}

\begin{prop} There are infinitely many elliptic curves with torsion $\mathbb{Z}/14\mathbb{Z}$ over $\mathbb{Q}(\sqrt{17}).$
\end{prop}

\begin{proof}
For the modular curve $$X_1(14): \: \: y^2+xy+y=x^3-x$$ we compute $$rank(X_1(14)(\mathbb{Q}(\sqrt{17})))=1.$$ Using similar reasoning to the one in Proposition $\ref{prop1}$ we conclude that there are infinitely many elliptic curves with torsion $\mathbb{Z}/14\mathbb{Z}$ over $\mathbb{Q}(\sqrt{17}).$

Also, one point of infinite order is given by $$\left(\dfrac{1}{2}(\sqrt{17}+3),-\sqrt{17}-5\right).$$
\end{proof}

\begin{prop} 
$\mathbb{Z}/15\mathbb{Z}$ cannot be a torsion group of an elliptic curve over $\mathbb{Q}(\sqrt{17}).$
\end{prop}

\begin{proof}
For the modular curve $$X_1(15): \:\: y^2+xy+y=x^3+x^2,$$ we compute $$rank(X_1(15)(\mathbb{Q}(\sqrt{17})))=0.$$ 
Hence, we only have to show that $$Y_1(15)(\mathbb{Q}(\sqrt{17}))=\emptyset,$$ i.e. that there are only cusps in $X_1(15)(\mathbb{Q}(\sqrt{17})).$

The $x-$coordinates of the cusps on $X_1(15)$ satisfy the equation $$x(x+1)(x^4+3x^3+4x^2+2x+1)(x^4-7x^3-6x^2+2x+1)=0.$$
So, the set of all cusps is $$X_1(15)(\mathbb{Q}(\sqrt{17})) \backslash Y_1(15)(\mathbb{Q}(\sqrt{17}))=\{ O, (0,0), (0,-1), (-1,0) \}.$$
We compute $$X_1(15)(\mathbb{Q}(\sqrt{17})) \cong \mathbb{Z}/4\mathbb{Z},$$ so now it is obvious that all points on the modular curve $X_1(15)$ defined over $\mathbb{Q}(\sqrt{17})$ are cusps. Therefore, $Y_1(15)(\mathbb{Q}(\sqrt{17}))=\emptyset,$ so there are no elliptic curves with torsion $\mathbb{Z}/15\mathbb{Z}$ over $\mathbb{Q}(\sqrt{17}).$
\end{proof}

\begin{prop} There are infinitely many elliptic curves with torsion $\mathbb{Z}/2\mathbb{Z} \oplus \mathbb{Z}/10\mathbb{Z}$ over $\mathbb{Q}(\sqrt{17}).$
\end{prop}

\begin{proof}
For the modular curve $$X_1(2,10): \:\: y^2=x^3+x^2-x,$$ we compute $$rank(X_1(2,10)(\mathbb{Q}(\sqrt{17})))=1.$$
One point of infinite order is $$\left(\sqrt{17}+4,3\sqrt{17}+12\right).$$
\end{proof}

\begin{prop} There are infinitely many elliptic curves with torsion $\mathbb{Z}/2\mathbb{Z} \oplus \mathbb{Z}/12\mathbb{Z}$ over $\mathbb{Q}(\sqrt{17}).$
\end{prop}

\begin{proof}
For the modular curve $$X_1(2,12): \:\: y^2=x^3-x^2+x,$$ we compute $$rank(X_1(2,12)(\mathbb{Q}(\sqrt{17})))=1.$$
A point of infinite order is $$\left(\dfrac{1}{2}(-\sqrt{17}+9),\dfrac{1}{2}(-3\sqrt{17}+19)\right).$$
\end{proof}

Now we have determined whether $\mathbb{Z}/n\mathbb{Z} \oplus \mathbb{Z}/n\mathbb{Z}$ is a possible torsion of an elliptic curve over $\mathbb{Q}(\sqrt{17}),$ for all modular functions $X_1(m,n)$ that are elliptic curves. To determine if the $\mathbb{Z}/n\mathbb{Z}, \:\: n=13, 16, 18,$ are possible torsion groups is somewhat more difficult, since corresponding modular curves are hyperelliptic curves.

The groups $\mathbb{Z}/n\mathbb{Z}, \:\: n=13, 18,$ are generally easier to deal with over quadratic fields, since we have the two following results that can be found in \cite{BBDN, KRUMM}.

\begin{thm} \label{tm1}
If $X_1(13)$ has a point defined over $\mathbb{Q}(\sqrt{d}),$ then:
\begin{enumerate}
   \item $d>0,$
   \item $d\equiv 1 \pmod{8}.$
\end{enumerate}
\end{thm}

\begin{thm} \label{tm2}
If $X_1(18)$ has a point defined over $\mathbb{Q}(\sqrt{d}),$ $d\neq -3,$ then:
\begin{enumerate}
   \item $d>0,$
   \item $d\equiv 1 \pmod{8},$
   \item $d\not\equiv 2 \pmod{3}.$
\end{enumerate}
\end{thm}

Now it becomes clear why we chose the field $\mathbb{Q}(\sqrt{17}),$ as we did not want to rule out the existence of the groups $\mathbb{Z}/n\mathbb{Z}, \:\: n=13, 18,$ as possible torsion subgroups.

\begin{prop} \label{tm3}
$\mathbb{Z}/13\mathbb{Z}$ is a possible torsion over $\mathbb{Q}(\sqrt{17}).$
\end{prop}

\begin{proof}
Let $J_1(13)$ be the Jacobian of the hyperelliptic curve
$$X_1(13): \: \: y^2=x^6-2x^5+x^4-2x^3+6x^2-4x+1$$ and let  $J_1^{17}(13)$ be its quadratic twist by 17, which becomes isomorphic to $J_1(13)$ over $\mathbb{Q}(\sqrt{17}).$
We compute $$rank(J_1(13)(\mathbb{Q}))=0,$$ $$rank(J_1^{17}(13)(\mathbb{Q}))=2.$$ Now, we have $$
rank(J_1(13)(\mathbb{Q}(\sqrt{17})))=rank(J_1(13)(\mathbb{Q}))+rank(J_1^{17}(13)(\mathbb{Q}))=2.$$

By searching for points on $X_1(13)(\mathbb{Q}(\sqrt{17}))$ in Magma, we find a point $\left(\dfrac{1}{2}, \dfrac{1}{8}\sqrt{17}\right)$ on the curve.

Since the $x-$coordinates of the cusps on the $X_1(13)$ are the solutions of the equation $$x(x-1)(x^3-4x^2+x+1)=0,$$
the cusps are $$X_1(13)(\mathbb{Q}(\sqrt{17})) \backslash Y_1(13)(\mathbb{Q}(\sqrt{17}))=\{\infty_+, \infty_-,(0,\pm 1),(1,\pm 1)\}.$$ We conclude that the point mentioned above is not a cusp and so the elliptic curve over $\mathbb{Q}(\sqrt{17})$ with a torsion subgroup $\mathbb{Z}/13\mathbb{Z}$ exists.
\end{proof}

Unlike in previous propositions, we do not have infinitely many elliptic curves with torsion $\mathbb{Z}/13\mathbb{Z},$ since by Falting's theorem the modular curve $X_1(13)$ can only have finitely many points over a number field.

\begin{prop} 
$\mathbb{Z}/16\mathbb{Z}$ cannot be a torsion group of an elliptic curve over $\mathbb{Q}(\sqrt{17}).$
\end{prop}

\begin{proof}
Let $J_1(16)$ be the Jacobian of the hyperelliptic curve $$X_1(16): \: \: y^2 = x(x^2+1)(x^2+2x-1)$$ and let  $J_1^{17}(16)$ be its quadratic twist.

We compute  $$rank(J_1(16)(\mathbb{Q}(\sqrt{17})))=rank(J_1(16)(\mathbb{Q}))+rank(J_1^{17}(16)(\mathbb{Q}))=0.$$

Since the rank is zero, we need to find the cusps in $X_1(16)(\mathbb{Q}(\sqrt{17}))$ and the torsion subgroup of $J_1(16)(\mathbb{Q}(\sqrt{17}))$ in order to determine if there exists a point on the Jacobian that arises from a point on the modular curve that is not a cusp.

As the $x-$coordinates of the cusps satisfy $$x(x -1)(x + 1)(x^2- 2x - 1)(x^2 + 2x-1)=0,$$ the cusps on $X_1(16)(\mathbb{Q}(\sqrt{17}))$ are $$X_1(16)(\mathbb{Q}(\sqrt{17})) \backslash Y_1(16)(\mathbb{Q}(\sqrt{17}))=\{\infty,(0,0),(1,\pm 2),(-1,\pm 2)\}.$$

We also compute $$J_1(16)(\mathbb{Q})_{tors} \cong \mathbb{Z}/2\mathbb{Z} \oplus \mathbb{Z}/10\mathbb{Z},$$
$$J_1^{17}(16)(\mathbb{Q})_{tors} \cong \mathbb{Z}/2\mathbb{Z} \oplus \mathbb{Z}/2\mathbb{Z}.$$

The set of points of odd order on the Jacobian $J_1(16)$ defined over $\mathbb{Q}(\sqrt{17})$ is $$J_1(16)(\mathbb{Q}(\sqrt{17}))_{(2')}\cong J_1(16)(\mathbb{Q})_{(2')} \oplus J_1^{17}(16)(\mathbb{Q})_{(2')}\cong \mathbb{Z}/5\mathbb{Z},$$

and the $2-$torsion subgroup is $$J_1(16)(\mathbb{Q}(\sqrt{17}))_{tors} \cong \mathbb{Z}/2\mathbb{Z} \oplus \mathbb{Z}/2\mathbb{Z}.$$ 

Thus, $J_1(16)(\mathbb{Q}) \cong \mathbb{Z}/2\mathbb{Z} \oplus \mathbb{Z}/10\mathbb{Z}$ and $J_1(16)(\mathbb{Q}) \cong J_1(16)(\mathbb{Q})_{tors}.$

In Magma, we find 20 divisor classes in Mumford representation \cite{CF}, 
$$(1, 0, 0), \:\:(x^2 + 2x + 1, 2x, 2), \:\:(x^2 + 2x + 1, -2x, 2),\:\: (x^2 - 2x + 1,
4x - 2, 2), $$ $$(x^2 - 2x + 1, -4x + 2, 2),\:\: (x + 1, 2, 1),\:\: (x + 1, -2, 1),\:\: (x, 0,
1), $$ $$ (x - 1, 2, 1),\:\: (x - 1, -2, 1),\:\: (x^2 + 2x - 1, 0, 2), \:\:(x^2 + x, 2x, 2),\:\:
(x^2 + x, -2x, 2), $$ $$(x^2 - 1, 2x, 2),\:\: (x^2 - 1, -2x, 2), \:\:(x^2 - 1, 2, 2),\:\: (x^2
    - 1, -2, 2), $$ $$(x^2 + 1, 0, 2), \:\:(x^2 - x, 2x, 2),\:\: (x^2 - x, -2x, 2).$$

The first divisor class represents the point at infinity, and for the rest of the divisor classes in Mumford representation, we follow the methods described in \cite{MAGMA} in order to retrieve the point on the Jacobian from its Mumford representation. 

For a triple $(a(x),b(x),d)$ in Mumford representation we define $A(x,z)$ as the homogenisation of the polynomial $a(x)$ of degree $d$ and $B(x,z)$ as the homogenisation of the polynomial $b(x)$ of degree $g+1,$ where $g$ is a genus of $X_1(16),$ in this case $g=2$.

Now, by solving the equations 
$$ A(x,z)=0, y=B(x,z), $$
we get the points $P_i=(x_i:y_i:1), \: i=1,...,d,$ in projective coordinates. Note that the number of points is exactly $d,$ as the polynomial $A(x,z)$ is of degree $d$.

The point on the Jacobian represented by $(a(x),b(x),d)$ is then the divisor class $$\left[P_1+...+P_d-d\infty\right],$$ if there is a single point $\infty$ at infinity, or  $$\left[P_1+...+P_d-\dfrac{d}{2}(\infty_+ + \infty_-)\right],$$ if there are two points $ \infty_+$ and $\infty_-$ at infinity.

For example, for the point $(x^2+2x+1,2x,2)$ on $J_1(16)(\mathbb{Q}(\sqrt{17}))$ in Mumford representation we have
$$A(x,z)=x^2+2xz+z^2,$$
$$B(x,z)=2xz^2,$$
and $P_1=P_2=(-1:-2:1),$ so we conclude that the point $(x^2+2x+1,2x,2)$ represents the divisor class $\left[(-1:-2:1)+(-1:-2:1)-2\infty\right]$ on the Jacobian $J_1(16)(\mathbb{Q}(\sqrt{17})).$

By doing so for every divisor class in Mumford representation, one can check that all divisor points correspond to the cusps in $X_1(16)(\mathbb{Q}(\sqrt{17})),$ so we conclude that $\mathbb{Z}/16\mathbb{Z}$ cannot be a torsion group of an elliptic curve over $\mathbb{Q}(\sqrt{17}).$
\end{proof}

\begin{prop} \label{tm4}
$\mathbb{Z}/18\mathbb{Z}$ cannot be a torsion group of an elliptic curve over $\mathbb{Q}(\sqrt{17}).$
\end{prop}

\begin{proof}

Let $J_1(18)$ be the Jacobian of the hyperelliptic curve $$X_1(18): \: \: y^2=x^6+2x^5+5x^4+10x^3+10x^2+4x+1$$ and let  $J_1^{17}(18)$ be its quadratic twist. We compute  $$
rank(J_1(18)(\mathbb{Q}(\sqrt{17})))=rank(J_1(18)(\mathbb{Q}))+rank(J_1^{17}(18)(\mathbb{Q}))=0.$$

The $x-$coordinates of the cusps in $X_1(18)$ satisfy the equation $$x(x+1)(x^2+x+1)(x^2-3x-1)=0,$$ so the cusps are $$X_1(18)(\mathbb{Q}(\sqrt{17})) \backslash Y_1(18)(\mathbb{Q}(\sqrt{17}))=\{\infty_+, \infty_-,(0,\pm 1),(-1,\pm 1)\}.$$

On the other hand, we compute $$J_1(18)(\mathbb{Q})_{tors} \cong \mathbb{Z}/21\mathbb{Z},$$
$$J_1^{17}(18)(\mathbb{Q})_{tors} \cong \{ O \}.$$ The set of points of odd order is $$J_1(18)(\mathbb{Q}(\sqrt{17}))_{(2')}\cong J_1(18)(\mathbb{Q})_{(2')} \oplus J_1^{17}(18)(\mathbb{Q})_{(2')}\cong \mathbb{Z}/21\mathbb{Z}.$$
As the polynomial $$f(x)=x^6+2x^5+5x^4+10x^3+10x^2+4x+1$$ has no zeros defined over $\mathbb{Q}(\sqrt{17}),$ we conclude that $J_1(18)(\mathbb{Q}(\sqrt{17}))$ has no points of order 2, and $$J_1(18)(\mathbb{Q}(\sqrt{17}))_{tors} \cong \mathbb{Z}/21\mathbb{Z}.$$
The elements of $J_1(18)(\mathbb{Q}(\sqrt{17}))_{tors}$ in Mumford representation are 

$$(1, 0, 0), \:\: (1, x^3 + x^2, 2), \:\: (1, -x^3 - x^2, 2),\:\:    	(x^2 + 2x + 1, x, 2),\:\: (x^2
    + 2x + 1, -x, 2),$$
$$(x^2, 2x + 1, 2), \:\:(x^2, -2x - 1, 2),\:\: (x + 1, x^3, 2),\:\: 
	(x + 1, -x^3, 2),$$ $$ (x + 1, x^3 + 2, 2),\:\:  (x + 1, -x^3 - 2, 2),\:\:
	(x, x^3 - 1, 2),\:\: (x, -x^3 + 1, 2),\:\:(x, x^3 + 1, 2),$$ 
$$(x, -x^3 - 1, 2), \:\:(x^2 + x, 2x + 1, 2),\:\: 
	(x^2
    + x, -2x - 1, 2)),\:\: (x^2 + x, 1, 2),$$ $$ (x^2 + x, -1, 2), \:\: (x^2 + x + 1, x - 1,
2), \:\: (x^2 + x + 1, -x + 1, 2), $$
	and one can easily conclude that all of the points correspond to the cusps, so we obtain our result.
\end{proof}

We proved the following theorem:

\begin{thm}\label{tm} The possible torsion subgroups of elliptic curves defined over $\mathbb{Q}(\sqrt{17})$ are the following:
$$\mathbb{Z}/n\mathbb{Z}, \:\:\: n=1,...,14,$$
$$\mathbb{Z}/2\mathbb{Z} \oplus \mathbb{Z}/2n\mathbb{Z}, \:\:\: n=1,...,6.$$
\end{thm}

\begin{remark}
	In \cite[Theorem 2.7.7, 2.7.8]{KRUMM} Krumm found a list of possible quadratic fields $\mathbb{Q}(\sqrt{d})$ over which torsion subgroups $\mathbb{Z}/13\mathbb{Z}$ and $\mathbb{Z}/18\mathbb{Z}$ may appear, for $0<d<1000.$
	
	In our case, for $0<d<100,$ there are only two such fields, $\mathbb{Q}(\sqrt{17}),$ over which $\mathbb{Z}/13\mathbb{Z}$ appears, and $\mathbb{Q}(\sqrt{33}),$ over which $\mathbb{Z}/18\mathbb{Z}$ appears.
	
	In this paper we were able to eliminate the rest of the fields for $0<d<100$ using only conditions from \ref{tm1} and \ref{tm2} and methods described in \ref{tm3} and \ref{tm4}.
\end{remark}

\section{Torsion over $\mathbb{Q}(\sqrt{d})$, $0<\lowercase{d}<100$}

For every quadratic field $\mathbb{Q}(\sqrt{d}),$ where $d$ is a non-negative squarefree integer $d\neq 1,$ $0<d<100,$ we found the torsion subgroups appearing by methods similar to the ones described in Theorem \ref{tm}.

Since we know that the groups from the Mazur's theorem appear as torsion subgroups, and the groups $\mathbb{Z}/3\mathbb{Z}\oplus \mathbb{Z}/3n\mathbb{Z}, \: n=1,2, \: \mathbb{Z}/4\mathbb{Z}\oplus \mathbb{Z}/4\mathbb{Z}$ do not appear as torsion subgroups of elliptic curves over mentioned fields, we give a list of possible torsion subgroups within the groups $$\mathbb{Z}/n\mathbb{Z}, \:\: n=11,13,14, 15, 16, 18, $$ $$\mathbb{Z}/2\mathbb{Z}\oplus \mathbb{Z}/2n\mathbb{Z}, n=5,6, \:\: $$ i.e. the rest of the possible 26 torsion subgroups over quadratic fields.

Note that the groups mentioned above appear as a torsion subgroup 37, 1, 38, 37, 6-17, 1, 35, 36 times, respectively, over all $\mathbb{Q}(\sqrt{d})$, where $0<d<100.$

We were unable to determine how many times exactly does the group $ \mathbb{Z}/16\mathbb{Z}$ appear as a torsion subgroup because of the following problem: computing (in Magma) the rank of the Jacobian of the modular curve $X_1(16)$ defined over the problematic quadratic fields listed in the table did not give a result, only the lower and the upper bound that were not the same. That was a problem since it is important to know the mentioned rank in order to know which method to use for determining whether that is a possible torsion subgroup or not.

Also, searching for the points on the modular curve $X_1(16)$ over the same fields did not yield a result.

\begin{table}[]
\centering
\caption{List of all possible torsion groups over quadratic fields $\mathbb{Q}(\sqrt{d}),$ for $0<d<100,$ without the groups from the Mazur's theorem}
\label{my-label}
\begin{tabular}{ll}
\hline
\multicolumn{1}{|l|}{Fields}                  & \multicolumn{1}{l|}{Possible torsion subgroups over a given field$\:\:\:\:\:\:\:\:\:\:\:\:\:\:\:\:\:\:\:\:\:\:\:\:\:\:$}                                                                                                                                               \\ \hline
\multicolumn{1}{|l|}{$\mathbb{Q}(\sqrt{2}) \:\:\:\:\:\:\:\:\:\:\:\:$}  & \multicolumn{1}{l|}{\begin{tabular}[c]{@{}l@{}}$\mathbb{Z}/n\mathbb{Z}, \:\:\: n=11$\end{tabular}}                         \\ \hline
\multicolumn{1}{|l|}{$\mathbb{Q}(\sqrt{3})$}  & \multicolumn{1}{l|}{\begin{tabular}[c]{@{}l@{}}$\mathbb{Z}/n\mathbb{Z}, \:\:\: n=14, 15$\\ $\mathbb{Z}/2\mathbb{Z} \oplus \mathbb{Z}/2n\mathbb{Z}, \:\:\: n=6$\end{tabular}}                 \\ \hline
\multicolumn{1}{|l|}{$\mathbb{Q}(\sqrt{5})$}  & \multicolumn{1}{l|}{\begin{tabular}[c]{@{}l@{}}$\mathbb{Z}/n\mathbb{Z}, \:\:\: n=15$\end{tabular}}                 \\ \hline
\multicolumn{1}{|l|}{$\mathbb{Q}(\sqrt{6})$}  & \multicolumn{1}{l|}{\begin{tabular}[c]{@{}l@{}}$\mathbb{Z}/n\mathbb{Z}, \:\:\: n=11, 14$\\ $\mathbb{Z}/2\mathbb{Z} \oplus \mathbb{Z}/2n\mathbb{Z}, \:\:\: n=5,6$\end{tabular}}                         \\ \hline
\multicolumn{1}{|l|}{$\mathbb{Q}(\sqrt{7})$}  & \multicolumn{1}{l|}{\begin{tabular}[c]{@{}l@{}}$\mathbb{Z}/n\mathbb{Z}, \:\:\: n=11, 14, 15$\\ $\mathbb{Z}/2\mathbb{Z} \oplus \mathbb{Z}/2n\mathbb{Z}, \:\:\: n=6$\end{tabular}}                     \\ \hline
\multicolumn{1}{|l|}{$\mathbb{Q}(\sqrt{10})$} & \multicolumn{1}{l|}{\begin{tabular}[c]{@{}l@{}}$\mathbb{Z}/n\mathbb{Z}, \:\:\: n=11, 14, 15, 16$\\ $\mathbb{Z}/2\mathbb{Z} \oplus \mathbb{Z}/2n\mathbb{Z}, \:\:\: n=5,6$\end{tabular}}                 \\ \hline
\multicolumn{1}{|l|}{$\mathbb{Q}(\sqrt{11})$} & \multicolumn{1}{l|}{\begin{tabular}[c]{@{}l@{}}$\mathbb{Z}/n\mathbb{Z}, \:\:\: n=11, 15$\end{tabular}}                         \\ \hline
\multicolumn{1}{|l|}{$\mathbb{Q}(\sqrt{13})$} & \multicolumn{1}{l|}{\begin{tabular}[c]{@{}l@{}}$\mathbb{Z}/n\mathbb{Z}, \:\:\: n=11, 15$\\ $\mathbb{Z}/2\mathbb{Z} \oplus \mathbb{Z}/2n\mathbb{Z}, \:\:\: n=5,6$\end{tabular}}                         \\ \hline
\multicolumn{1}{|l|}{$\mathbb{Q}(\sqrt{14})$} & \multicolumn{1}{l|}{\begin{tabular}[c]{@{}l@{}}$\mathbb{Z}/n\mathbb{Z}, \:\:\: n=14, 15$\\ $\mathbb{Z}/2\mathbb{Z} \oplus \mathbb{Z}/2n\mathbb{Z}, \:\:\: n=5$\end{tabular}}                 \\ \hline
\multicolumn{1}{|l|}{$\mathbb{Q}(\sqrt{15})$} & \multicolumn{1}{l|}{\begin{tabular}[c]{@{}l@{}}$\mathbb{Z}/n\mathbb{Z}, \:\:\: n=15, 16$\\ $\mathbb{Z}/2\mathbb{Z} \oplus \mathbb{Z}/2n\mathbb{Z}, \:\:\: n=5$\end{tabular}}             \\ \hline
\multicolumn{1}{|l|}{$\mathbb{Q}(\sqrt{17})$} & \multicolumn{1}{l|}{\begin{tabular}[c]{@{}l@{}}$\mathbb{Z}/n\mathbb{Z}, \:\:\: n=11,13,14$\\ $\mathbb{Z}/2\mathbb{Z} \oplus \mathbb{Z}/2n\mathbb{Z}, \:\:\: n=5,6$\end{tabular}}                             \\ \hline
\multicolumn{1}{|l|}{$\mathbb{Q}(\sqrt{19})$} & \multicolumn{1}{l|}{\begin{tabular}[c]{@{}l@{}}$\mathbb{Z}/n\mathbb{Z}, \:\:\: n=11, 14$\\ $\mathbb{Z}/2\mathbb{Z} \oplus \mathbb{Z}/2n\mathbb{Z}, \:\:\: n=5$\end{tabular}}                         \\ \hline
\multicolumn{1}{|l|}{$\mathbb{Q}(\sqrt{21})$} & \multicolumn{1}{l|}{\begin{tabular}[c]{@{}l@{}}$\mathbb{Z}/n\mathbb{Z}, \:\:\: n=11$\\ $\mathbb{Z}/2\mathbb{Z} \oplus \mathbb{Z}/2n\mathbb{Z}, \:\:\: n=6$\end{tabular}}                         \\ \hline
\multicolumn{1}{|l|}{$\mathbb{Q}(\sqrt{22})$} & \multicolumn{1}{l|}{\begin{tabular}[c]{@{}l@{}}$\mathbb{Z}/n\mathbb{Z}, \:\:\: n=11, 14, 15$\\ $\mathbb{Z}/2\mathbb{Z} \oplus \mathbb{Z}/2n\mathbb{Z}, \:\:\: n=6$\end{tabular}}                     \\ \hline
\multicolumn{1}{|l|}{$\mathbb{Q}(\sqrt{23})$} & \multicolumn{1}{l|}{\begin{tabular}[c]{@{}l@{}} $\mathbb{Z}/2\mathbb{Z} \oplus \mathbb{Z}/2n\mathbb{Z}, \:\:\: n=6$\end{tabular}}                     \\ \hline
\multicolumn{1}{|l|}{$\mathbb{Q}(\sqrt{26})$} & \multicolumn{1}{l|}{\begin{tabular}[c]{@{}l@{}}$\mathbb{Z}/n\mathbb{Z}, \:\:\: n=14, 15, \text{ maybe } \mathbb{Z}/16\mathbb{Z}$\\ $\mathbb{Z}/2\mathbb{Z} \oplus \mathbb{Z}/2n\mathbb{Z}, \:\:\: n=5$
\end{tabular}} \\ \hline
\multicolumn{1}{|l|}{$\mathbb{Q}(\sqrt{29})$} & \multicolumn{1}{l|}{\begin{tabular}[c]{@{}l@{}}$\mathbb{Z}/n\mathbb{Z}, \:\:\: n=11, 14, 15$\\ $\mathbb{Z}/2\mathbb{Z} \oplus \mathbb{Z}/2n\mathbb{Z}, \:\:\: n=5$\end{tabular}}                     \\ \hline
\multicolumn{1}{|l|}{$\mathbb{Q}(\sqrt{30})$} & \multicolumn{1}{l|}{\begin{tabular}[c]{@{}l@{}}$\mathbb{Z}/n\mathbb{Z}, \:\:\: n=11, 15$\\ $\mathbb{Z}/2\mathbb{Z} \oplus \mathbb{Z}/2n\mathbb{Z}, \:\:\: n=5,6$\end{tabular}}                     \\ \hline
\multicolumn{1}{|l|}{$\mathbb{Q}(\sqrt{31})$} & \multicolumn{1}{l|}{\begin{tabular}[c]{@{}l@{}}$\mathbb{Z}/n\mathbb{Z}, \:\:\: n=14, \text{ maybe } \mathbb{Z}/16\mathbb{Z} $\\ $\mathbb{Z}/2\mathbb{Z} \oplus \mathbb{Z}/2n\mathbb{Z}, \:\:\: n=5$ 
\end{tabular}}     \\ \hline
\multicolumn{1}{|l|}{$\mathbb{Q}(\sqrt{33})$} & \multicolumn{1}{l|}{\begin{tabular}[c]{@{}l@{}}$\mathbb{Z}/n\mathbb{Z}, \:\:\: n=11, 14, 15, 18$\\ $\mathbb{Z}/2\mathbb{Z} \oplus \mathbb{Z}/2n\mathbb{Z}, \:\:\: n=5$\end{tabular}}                 \\ \hline
\multicolumn{1}{|l|}{$\mathbb{Q}(\sqrt{34})$} & \multicolumn{1}{l|}{\begin{tabular}[c]{@{}l@{}}$\mathbb{Z}/n\mathbb{Z}, \:\:\: n=14, 15$\\ $\mathbb{Z}/2\mathbb{Z} \oplus \mathbb{Z}/2n\mathbb{Z}, \:\:\: n=5,6$\end{tabular}}                 \\ \hline
                                  
\end{tabular}
\end{table}

\begin{table}[]
\centering
\label{my-label}
\begin{tabular}{|l|l|}
\hline
Fields                  & Possible torsion subgroups over a given field$\:\:\:\:\:\:\:\:\:\:\:\:\:\:\:\:\:\:\:\:\:\:\:\:\:\:$                                                                                                                                       \\ \hline
$\mathbb{Q}(\sqrt{35})\:\:\:\:\:\:\:\:\:\:\:\:$ & \begin{tabular}[c]{@{}l@{}}$\mathbb{Z}/n\mathbb{Z}, \:\:\: n=11, 14$\\ $\mathbb{Z}/2\mathbb{Z} \oplus \mathbb{Z}/2n\mathbb{Z}, \:\:\: n=5,6$\end{tabular}                 \\ \hline
$\mathbb{Q}(\sqrt{37})$ & \begin{tabular}[c]{@{}l@{}}$\mathbb{Z}/n\mathbb{Z}, \:\:\: n=14, 15$\\ $\mathbb{Z}/2\mathbb{Z} \oplus \mathbb{Z}/2n\mathbb{Z}, \:\:\: n=5,6$\end{tabular}         \\ \hline
$\mathbb{Q}(\sqrt{38})$ & \begin{tabular}[c]{@{}l@{}}$\mathbb{Z}/n\mathbb{Z}, \:\:\: n=14, 15$\end{tabular}         \\ \hline
$\mathbb{Q}(\sqrt{39})$ & \begin{tabular}[c]{@{}l@{}}$\mathbb{Z}/n\mathbb{Z}, \:\:\: n=11$\\ $\mathbb{Z}/2\mathbb{Z} \oplus \mathbb{Z}/2n\mathbb{Z}, \:\:\: n=5,6$\end{tabular}                 \\ \hline
$\mathbb{Q}(\sqrt{41})$ & \begin{tabular}[c]{@{}l@{}}$\mathbb{Z}/n\mathbb{Z}, \:\:\: n=11, 14, 15, 16$\\ $\mathbb{Z}/2\mathbb{Z} \oplus \mathbb{Z}/2n\mathbb{Z}, \:\:\: n=6$\end{tabular}         \\ \hline
$\mathbb{Q}(\sqrt{42})$ & \begin{tabular}[c]{@{}l@{}}$\mathbb{Z}/n\mathbb{Z}, \:\:\: n=14, 15$\\ $\mathbb{Z}/2\mathbb{Z} \oplus \mathbb{Z}/2n\mathbb{Z}, \:\:\: n=6$\end{tabular}         \\ \hline
$\mathbb{Q}(\sqrt{43})$ & \begin{tabular}[c]{@{}l@{}}$\mathbb{Z}/n\mathbb{Z}, \:\:\: n=11, 15$\end{tabular}             \\ \hline
$\mathbb{Q}(\sqrt{46})$ & \begin{tabular}[c]{@{}l@{}}$\mathbb{Z}/n\mathbb{Z}, \:\:\: n=11$\\ $\mathbb{Z}/2\mathbb{Z} \oplus \mathbb{Z}/2n\mathbb{Z}, \:\:\: n=5,6$\end{tabular}                 \\ \hline
$\mathbb{Q}(\sqrt{47})$ & \begin{tabular}[c]{@{}l@{}}$\mathbb{Z}/n\mathbb{Z}, \:\:\: n=14, \text{ maybe } \mathbb{Z}/16\mathbb{Z}$\\ $\mathbb{Z}/2\mathbb{Z} \oplus \mathbb{Z}/2n\mathbb{Z}, \:\:\: n=6$
\end{tabular}     \\ \hline
$\mathbb{Q}(\sqrt{51})$ & \begin{tabular}[c]{@{}l@{}}$\mathbb{Z}/n\mathbb{Z}, \:\:\: n=11, 14, 16$\\ $\mathbb{Z}/2\mathbb{Z} \oplus \mathbb{Z}/2n\mathbb{Z}, \:\:\: n=5$\end{tabular}             \\ \hline
$\mathbb{Q}(\sqrt{53})$ & \begin{tabular}[c]{@{}l@{}}$\mathbb{Z}/n\mathbb{Z}, \:\:\: n=14$\\ $\mathbb{Z}/2\mathbb{Z} \oplus \mathbb{Z}/2n\mathbb{Z}, \:\:\: n=5$\end{tabular}             \\ \hline
$\mathbb{Q}(\sqrt{55})$ & \begin{tabular}[c]{@{}l@{}}$\mathbb{Z}/n\mathbb{Z}, \:\:\: n=11, 14, 15$\\ $\mathbb{Z}/2\mathbb{Z} \oplus \mathbb{Z}/2n\mathbb{Z}, \:\:\: n=5,6$\end{tabular}             \\ \hline
$\mathbb{Q}(\sqrt{57})$ & \begin{tabular}[c]{@{}l@{}}$\mathbb{Z}/n\mathbb{Z}, \:\:\: n=11, 15$\\ $\mathbb{Z}/2\mathbb{Z} \oplus \mathbb{Z}/2n\mathbb{Z}, \:\:\: n=5$\end{tabular}             \\ \hline
$\mathbb{Q}(\sqrt{58})$ & \begin{tabular}[c]{@{}l@{}}$\mathbb{Z}/n\mathbb{Z}, \:\:\: n=11,15, \text{ maybe } \mathbb{Z}/16\mathbb{Z}$\\ $\mathbb{Z}/2\mathbb{Z} \oplus \mathbb{Z}/2n\mathbb{Z}, \:\:\: n=6$\end{tabular}         \\ \hline
$\mathbb{Q}(\sqrt{59})$ & \begin{tabular}[c]{@{}l@{}}$\mathbb{Z}/n\mathbb{Z}, \:\:\: n=14, 15$\\ $\mathbb{Z}/2\mathbb{Z} \oplus \mathbb{Z}/2n\mathbb{Z}, \:\:\: n=5,6$\end{tabular}         \\ \hline
$\mathbb{Q}(\sqrt{61})$ & \begin{tabular}[c]{@{}l@{}}$\mathbb{Z}/n\mathbb{Z}, \:\:\: n=11$\\ $\mathbb{Z}/2\mathbb{Z} \oplus \mathbb{Z}/2n\mathbb{Z}, \:\:\: n=6$\end{tabular}                 \\ \hline
$\mathbb{Q}(\sqrt{62})$ & \begin{tabular}[c]{@{}l@{}}$\mathbb{Z}/n\mathbb{Z}, \:\:\: n=11, 14, \text{ maybe } \mathbb{Z}/16\mathbb{Z}$ \end{tabular}         \\ \hline
$\mathbb{Q}(\sqrt{65})$ & \begin{tabular}[c]{@{}l@{}}$\mathbb{Z}/n\mathbb{Z}, \:\:\: n=11$\\ $\mathbb{Z}/2\mathbb{Z} \oplus \mathbb{Z}/2n\mathbb{Z}, \:\:\: n=6$\end{tabular}                 \\ \hline
$\mathbb{Q}(\sqrt{66})$ & \begin{tabular}[c]{@{}l@{}}$\mathbb{Z}/n\mathbb{Z}, \:\:\: n=11, 14$\\ $\mathbb{Z}/2\mathbb{Z} \oplus \mathbb{Z}/2n\mathbb{Z}, \:\:\: n=5,6$\end{tabular}                 \\ \hline
$\mathbb{Q}(\sqrt{67})$ & \begin{tabular}[c]{@{}l@{}}$\mathbb{Z}/n\mathbb{Z}, \:\:\: n=15$\end{tabular}         \\ \hline
$\mathbb{Q}(\sqrt{69})$ & \begin{tabular}[c]{@{}l@{}} $\mathbb{Z}/2\mathbb{Z} \oplus \mathbb{Z}/2n\mathbb{Z}, \:\:\: n=6$\end{tabular}             \\ \hline
$\mathbb{Q}(\sqrt{70})$ & \begin{tabular}[c]{@{}l@{}}$\mathbb{Z}/n\mathbb{Z}, \:\:\: n=14, 15, 16$\\ $\mathbb{Z}/2\mathbb{Z} \oplus \mathbb{Z}/2n\mathbb{Z}, \:\:\: n=5,6$\end{tabular} 
\\ \hline
$\mathbb{Q}(\sqrt{71})$ & \begin{tabular}[c]{@{}l@{}}$\mathbb{Z}/n\mathbb{Z}, \:\:\: n=14, 15$\\ $\mathbb{Z}/2\mathbb{Z} \oplus \mathbb{Z}/2n\mathbb{Z}, \:\:\: n=5,6$\end{tabular}         \\ \hline

\end{tabular}
\end{table}

\begin{table}[]
\centering
\label{my-label}
\begin{tabular}{|l|l|}
\hline
Fields                  & Possible torsion subgroups over a given field$\:\:\:\:\:\:\:\:\:\:\:\:\:\:\:\:\:\:\:\:\:\:\:\:\:\:$                                                                                                                                       \\ \hline
$\mathbb{Q}(\sqrt{73})\:\:\:\:\:\:\:\:\:\:\:\:$ & \begin{tabular}[c]{@{}l@{}}$\mathbb{Z}/n\mathbb{Z}, \:\:\: n=11, 14, 15$\\ $\mathbb{Z}/2\mathbb{Z} \oplus \mathbb{Z}/2n\mathbb{Z}, \:\:\: n=5,6$\end{tabular}             \\ \hline
$\mathbb{Q}(\sqrt{74})$ & \begin{tabular}[c]{@{}l@{}}$\mathbb{Z}/n\mathbb{Z}, \:\:\: n=11, 15, \text{ maybe } \mathbb{Z}/16\mathbb{Z}$\\ $\mathbb{Z}/2\mathbb{Z} \oplus \mathbb{Z}/2n\mathbb{Z}, \:\:\: n=5$
\end{tabular}     \\ \hline
$\mathbb{Q}(\sqrt{77})$ & \begin{tabular}[c]{@{}l@{}}$\mathbb{Z}/n\mathbb{Z}, \:\:\: n=11$\\ $\mathbb{Z}/2\mathbb{Z} \oplus \mathbb{Z}/2n\mathbb{Z}, \:\:\: n=5$\end{tabular}                 \\ \hline
$\mathbb{Q}(\sqrt{78})$ & \begin{tabular}[c]{@{}l@{}}$\mathbb{Z}/n\mathbb{Z}, \:\:\: n=15, \text{ maybe } \mathbb{Z}/16\mathbb{Z}$\\ $\mathbb{Z}/2\mathbb{Z} \oplus \mathbb{Z}/2n\mathbb{Z}, \:\:\: n=6$\end{tabular} \\ \hline
$\mathbb{Q}(\sqrt{79})$ & \begin{tabular}[c]{@{}l@{}}$\mathbb{Z}/n\mathbb{Z}, \:\:\: n=11, 14, \text{ maybe } \mathbb{Z}/16\mathbb{Z}$\\  $\mathbb{Z}/2\mathbb{Z} \oplus \mathbb{Z}/2n\mathbb{Z}, \:\:\: n=5$ \end{tabular}         \\ \hline
$\mathbb{Q}(\sqrt{82})$ & \begin{tabular}[c]{@{}l@{}}$\mathbb{Z}/n\mathbb{Z}, \:\:\: n=14, 15, \text{ maybe } \mathbb{Z}/16\mathbb{Z}$\\ $\mathbb{Z}/2\mathbb{Z} \oplus \mathbb{Z}/2n\mathbb{Z}, \:\:\: n=6$ \end{tabular} \\ \hline
$\mathbb{Q}(\sqrt{83})$ & \begin{tabular}[c]{@{}l@{}}$\mathbb{Z}/n\mathbb{Z}, \:\:\: n=11, 14$\\ $\mathbb{Z}/2\mathbb{Z} \oplus \mathbb{Z}/2n\mathbb{Z}, \:\:\: n=6$\end{tabular}                 \\ \hline
$\mathbb{Q}(\sqrt{85})$ & \begin{tabular}[c]{@{}l@{}}$\mathbb{Z}/n\mathbb{Z}, \:\:\: n=11, 14, 15$\\ $\mathbb{Z}/2\mathbb{Z} \oplus \mathbb{Z}/2n\mathbb{Z}, \:\:\: n=6$\end{tabular}             \\ \hline
$\mathbb{Q}(\sqrt{86})$ & \begin{tabular}[c]{@{}l@{}}$\mathbb{Z}/n\mathbb{Z}, \:\:\: n=11, 15$\\ $\mathbb{Z}/2\mathbb{Z} \oplus \mathbb{Z}/2n\mathbb{Z}, \:\:\: n=5$\end{tabular}             \\ \hline
$\mathbb{Q}(\sqrt{87})$ & \begin{tabular}[c]{@{}l@{}}$\mathbb{Z}/n\mathbb{Z}, \:\:\: n=11, 14, 15, \text{ maybe } \mathbb{Z}/16\mathbb{Z}$ \end{tabular}     \\ \hline
$\mathbb{Q}(\sqrt{89})$ & \begin{tabular}[c]{@{}l@{}}$\mathbb{Z}/n\mathbb{Z}, \:\:\: n=14, 15$\\ $\mathbb{Z}/2\mathbb{Z} \oplus \mathbb{Z}/2n\mathbb{Z}, \:\:\: n=5,6$\end{tabular}         \\ \hline
$\mathbb{Q}(\sqrt{91})$ & \begin{tabular}[c]{@{}l@{}}$\mathbb{Z}/n\mathbb{Z}, \:\:\: n=11, 14, 15$\\ $\mathbb{Z}/2\mathbb{Z} \oplus \mathbb{Z}/2n\mathbb{Z}, \:\:\: n=5$\end{tabular}             \\ \hline
$\mathbb{Q}(\sqrt{93})$ & \begin{tabular}[c]{@{}l@{}}$\mathbb{Z}/n\mathbb{Z}, \:\:\: n=14, 15, 16$\\ $\mathbb{Z}/2\mathbb{Z} \oplus \mathbb{Z}/2n\mathbb{Z}, \:\:\: n=5,6$\end{tabular}     \\ \hline
$\mathbb{Q}(\sqrt{94})$ & \begin{tabular}[c]{@{}l@{}}$\mathbb{Z}/n\mathbb{Z}, \:\:\: n=11, 14, \text{ maybe } \mathbb{Z}/16\mathbb{Z}$\\ $\mathbb{Z}/2\mathbb{Z} \oplus \mathbb{Z}/2n\mathbb{Z}, \:\:\: n=5,6$\end{tabular}         \\ \hline
$\mathbb{Q}(\sqrt{95})$ & \begin{tabular}[c]{@{}l@{}}$\mathbb{Z}/n\mathbb{Z}, \:\:\: n=11$\\ $\mathbb{Z}/2\mathbb{Z} \oplus \mathbb{Z}/2n\mathbb{Z}, \:\:\: n=5,6$\end{tabular}                 \\ \hline
$\mathbb{Q}(\sqrt{97})$ & \begin{tabular}[c]{@{}l@{}}$\mathbb{Z}/n\mathbb{Z}, \:\:\: n=14, 15$\\ $\mathbb{Z}/2\mathbb{Z} \oplus \mathbb{Z}/2n\mathbb{Z}, \:\:\: n=5$\end{tabular}         \\ \hline
\end{tabular}
 \\ 
\end{table}

\clearpage
\bibliographystyle{unsrt}

\begin{thebibliography}{}

  
   
\bibitem{HB}
  H. Baaziz,.
  \textit{Equations for the modular curve $X_{1}(N)$ and models of elliptic
curves with torsion points},
   Math. Comp. 79 (2010), 2371-2386.  
   
\bibitem{MAGMA}
  W. Bosma, J. Cannon, C. Playoust
  \textit{The Magma algebra system. I. The
user language},
   J. Symbolic Comput. 24 (1997), 235-265.
   
\bibitem{BBDN}
    J. Bosman, P. Bruin, A. Dujella, F. Najman,
   \textit{Ranks of elliptic curves with prescribed torsion over number fields},
   Int. Math. Res. Not. IMRN 2014 (2014), 2885-2923.
   
\bibitem{CF}
  H. Cohen, G. Frey,
  \textit{Handbook of Elliptic and Hyperelliptic Curve Cryptography},
  Chapman \& Hall/CRC, 2006.

\bibitem{SK}
     S. Kamienny,
  \textit{Torsion points on elliptic curves and q-coefficients of modular
forms},
 Invent. Math. 109 (1992), 221-–229.
 
  \bibitem{SKFN}
  S. Kamienny, F. Najman,
  \textit{Torsion groups of elliptic curves over
quadratic fields},
 Acta. Arith. 152 (2012), 291–-305.
 
 \bibitem{MAKFM}
  M. A. Kenku, F. Momose,
  \textit{Torsion points on elliptic curves defined over
quadratic fields},
 Nagoya Math. J. 109, (1988), 125–-149.
 
\bibitem{KRUMM}
  D. Krumm,
  \textit{Quadratic Points on Modular Curves},
  Doctoral thesis, Athens,
Georgia, 2013.  

\bibitem{BM}
  B. Mazur,
  \textit{Modular curves and the Eisenstein ideal},
  Inst. Hautes Etudes Sci. Publ. Math. 47 (1978), 33-–186.

\bibitem{FN1}
  F. Najman,
  \textit{Complete classification of torsion of elliptic curves
over quadratic cyclotomic fields},
  J. Number Theory 130 (2010), 1964-1968.  
   
\bibitem{FN2}
  F. Najman,
  \textit{Torsion of elliptic curves over quadratic
cyclotomic fields},
  Math. J. Okayama Univ. 53 (2011), 75-82.  
   
\bibitem{FPR}
  F. P. Rabarison,
  \textit{Structure de torsion des courbes elliptiques sur les
corps quadratiques},
   Acta Arith. 144 (2010), 17-52.
  
\bibitem{DJZ}
  D. J. Zywina,
  \textit{On the possible images of the mod $l$ representations associated to elliptic curves over $\mathbb{Q}$},
  	preprint. 
  
\end{thebibliography}

\end{document}